\documentclass{amsart}

\RequirePackage{fix-cm}
\usepackage{graphicx}
\usepackage{amsmath, amsopn,amstext,amscd,amsfonts,amssymb}
\usepackage{dsfont}
\usepackage{comment}
\usepackage[active]{srcltx}
\usepackage{graphicx, epsfig, subfig}

\usepackage{textcomp}

\usepackage{algorithm}

\makeatletter
\def\BState{\State\hskip-\ALG@thistlm}
\makeatother

\def\downbar#1{
\setbox10=\hbox{$#1$}
            \dimen10=\ht10 \advance\dimen10 by 2.5pt
            \ifdim \dimen10<15pt %equals approximately 0.5cm
               \advance\dimen10 by -0.5pt
               \dimen11=\dimen10
               \advance\dimen10 by 2.5pt
               \lower \dimen11
            \else \lower \ht10 \fi
            \hbox {\hskip 1.5pt \vrule height \dimen10 depth \dp10}}
\def\upbar#1{
\setbox10=\hbox{$#1$}
            \dimen10=\ht10 \advance\dimen10 by \dp10 \advance\dimen10 by 2.5pt
            \ifdim \dimen10<15pt %equals approximately 0.5cm
                \advance\dimen10 by 2pt \fi
            \raise 2.5pt \hbox {\hskip -1.5pt \vrule height \dimen10}}

\usepackage{multicol}
\usepackage{colortbl}
\usepackage{exscale}
\usepackage{amssymb,latexsym,amsthm,amsfonts,color,fancyhdr,mathrsfs}
\usepackage{lscape}
\usepackage{rotating}
\usepackage{caption}

\newtheorem{theorem}{\bf Theorem}[section]
\newtheorem{proposition}{\bf Proposition}[section]
\newtheorem{lemma}{\bf Lemma}[section]

\newtheorem{remark}{\bf Remark}[section]

\numberwithin{equation}{section}
\begin{document}
\title[Appell-type orthogonal polynomials on lattices]{Some Appell-type orthogonal polynomials on lattices}

%%%%%%%%%%%%%%%%%%%%%%%%%%%%%%%%%%%%%%%%%%%%%%%%%%%%%%%%%%%%%%%%%%%%%%%%%%%%%
\author{D. Mbouna}
\address{D. Mbouna \\Department of Mathematics, Faculty of Sciences, University of Porto, Campo Alegre st., 687, 4169-007 Porto, Portugal}
\email{dieudonne.mbouna@fc.up.pt}

%\author{R. \'Alvarez-Nodarse}
%\address{University of Coimbra, CMUC, Dep. Mathematics, 3001-501 Coimbra, Portugal}
%\email{josep@mat.uc.pt}
%\author{K. Castillo}
%
%\email{kenier@mat.uc.pt}
%\author{D. Mbouna}
\author{A. Suzuki}
\address{A. Suzuki \\University of Coimbra, CMUC, Dep. Mathematics, 3001-501 Coimbra, Portugal}
\email{asuzuki@uc.pt}
%%%%%%%%%%%%%%%%%%%%%%%%%%%%%%%%%%%%%%%%%%%%%%%%%%%%%%%%%%%%%%%%%%%%%%%%%%%%%

\subjclass[2010]{42C05, 33C45}
\date{\today}
\keywords{Al-Salam-Chihara polynomials, lattice, Appell orthogonal polynomial}

\maketitle
\begin{abstract}
We investigate on some Appell-type orthogonal polynomial sequences on $q$-quadratic lattices and we provide some entire new characterizations of the Al-Salam-Chihara polynomials (including the Rogers $q$-Hermite polynomials). The corresponding regular forms are well described. The proposed method can be applied to similar and to more general problems involving the Askey-Wilson and the Averaging operators, in order to obtain new characterization theorems for classical and semiclassical orthogonal polynomials on lattices.
\end{abstract}

%
%In this chapter we show that the only sequences of orthogonal polynomials $(P_n)_{n\geq 0}$ satisfying
%\begin{align*}
%\phi(x)\mathcal{D}_q P_{n}(x)=a_n\mathcal{S}_q P_{n+1}(x) +b_n\mathcal{S}_q P_n(x) +c_n\mathcal{S}_q P_{n-1}(x),
%\end{align*}
%($c_n\neq 0$) where $\phi$ is a well chosen polynomial of degree at most two, $\mathcal{D}_q$ is the Askey-Wilson operator and $\mathcal{S}_q$ the averaging operator, are the multiple of Askey-Wilson polynomials, or specific or limiting cases of them. We also provide other characterization of such orthogonal polynomial sequences.
%

\section{Introduction}\label{introduction}
Classical orthogonal polynomial sequences (OPS) are certainly the most studied ones. This class of OPS has some beautiful properties and characterizations as well as applications in other related fields (number theory, probability, mathematical physics, approximation theory and many others mathematics branches). For instance, their $\textit{derivatives}$ are also OPS. One special case of this family is the situation where OPS and their $\textit{derivatives}$ coincide: this is known in the literature as Appell OPS. This notion was introduced in 1880 in a work by P. Appell \cite{PA1880}. That is the problem of finding polynomial sequences, $(f_n)_{n\geq 0}$, for which the following equation holds
\begin{align}\label{appell-general-equation-intro}
\mathrm{D} f_n(x)=r_nf_{n-1}(x) \quad (n=0,1,\ldots)\;,
\end{align} 
where $(r_n)_{n\geq 0}$ is a nonzero complex sequence of numbers and $\mathrm{D}$ is a lowering operator (this means an operator reducing by one the degree of any polynomial sequence). Since that time, all polynomial sequences with property \eqref{appell-general-equation-intro} are called Appell sequences (see \cite{A-1967, AS1954}). Along this work, we will focus only on Appell OPS.

We recall that if $\mathrm{D} =d/dx$ in \eqref{appell-general-equation-intro}, then the corresponding orthogonal polynomial sequence is the Hermite polynomial (see \cite{Al-Salam-1972}). If $\mathrm{D}$ is replaced by the $q$-Jackson operator $D_q$ (respectively the Hahn operator $D_{q,\omega}$) defined by 
\begin{align*}
D_{q,\omega}f(x)=\frac{f(qx+\omega)-f(x)}{(q-1)x+\omega},\quad 0<q<1,~\omega \in \mathbb{C}\;,
\end{align*}  
where $D_{q}=D_{q,0}$, then the corresponding Appell OPS are, up to an affine transformation of the variable, the Al-Salam-Carlitz polynomials (see \cite{RKDP2020, Griff2006}). In \cite{ALPM2008} it is studied the case of \eqref{appell-general-equation-intro} where 
$$ \mathrm{D}=2\frac{d}{dx}\; x\;\frac{d}{dx}  +\epsilon \frac{d}{dx},\quad \epsilon=\pm 1\;,$$ providing then a new characterization of the Laguerre polynomials. Such OPS received considerable attention along the last decade and since that time. Now consider the Askey-Wilson operator, $\mathcal{D}_q$, which is defined by
\begin{align*}%\label{AWxs}
\mathcal{D}_q\, p(x(s))= \frac{p(x(s+1/2))-p(x(s-1/2))}{x(s+1/2)-x(s-1/2)},\quad x(s)=\mbox{$\frac12$}(q^s +q^{-s})\;,
\end{align*}
for every polynomial $p$. We assume that $0<q<1$. (Taking $q^s=e^{i \theta}$ we recover $\mathcal{D}_q$ as defined in  \cite[(21.6.2)]{I2005}.) We define the averaging operator by
\begin{align*}%\label{AWxs}
\mathcal{S}_q\, p(x(s))= \frac{1}{2}\Big( p(x(s+1/2))+p(x(s-1/2))\Big),\quad x(s)=\mbox{$\frac12$}(q^s +q^{-s})\;.
\end{align*}

The problem of finding OPS solutions of \eqref{appell-general-equation-intro} whenever $\mathrm{D}=\mathcal{D}_q$ appeared as a special case of a problem posed by M. Ismail in \cite[Conjecture 24.7.8]{I2005}. This case of \eqref{appell-general-equation-intro} was firstly solved by W. Al-Salam in \cite{A-1995} and secondly by J. Galiffa and W. Ong  in \cite{DB2014} using different methods and characterising the Rogers $q$-Hermite polynomials as the only solutions. Despite this, none of the methods used in both works could be useful to solve the conjecture \cite[Conjecture 24.7.8]{I2005} in its entire form. This is only due the complexity of the Askey-Wilson operator and its properties. Recently in \cite{KDPconj}, the authors addressed this conjecture in its entire form using some new techniques. In addition, a situation of \eqref{appell-general-equation-intro} where operators $\mathcal{D}_q$ and $\mathcal{S}_q$ are both involved as the following equation
\begin{align*}
\mathcal{D}_qf_n(x)=r_n\mathcal{S}_qf_{n-1}(x)\quad (n=0,1,\ldots)\;,
\end{align*}
is considered in \cite{KCDMJP2021-a} characterizing some special cases of the Askey-Wilson polynomials. The purpose of this work is to solve \eqref{appell-general-equation-intro} for operators $\mathrm{D}=\mathcal{S}_q\mathcal{D}_q$ and $\mathrm{D}=\mathcal{D}_q\mathcal{S}_q$. This leads to a new characterization of the Al-Salam-Chihara polynomials. In addition we also characterize the corresponding regular forms. This definitely provides some ideas on polynomial bases to use when dealing with problems with the averaging and the Askey-Wilson operators. The method suggests a description/study of semiclassical OPS by deriving a system of difference or differential equations satisfied by their coefficients of the three term recurrence relation, since from this, some asymptotic behaviours and/or full expressions of these coefficients can be obtained. This approach is also presented in \cite{CFZ2019, DFS2021, FR2018} and some references therein, where authors used Painlev\'e equations to study the differential (or difference) equations satisfied by the coefficients of the three term recurrence relation of semiclassical OPS.

The structure of the paper is as follows. Section 2 presents some basic facts of the algebraic theory of OPS together with some useful results. Sections \ref{first-case-main-results} and \ref{second-case-main-results}  contain our main results for each situation.

\section{Background and preliminary results}
Our approach is based upon the algebraic theory of orthogonal polynomials developed by P. Maroni (see \cite{M1991}). Let $\mathcal{P}$ be the vector space of all polynomials with complex coefficients and let $\mathcal{P}^*$ be its algebraic dual. Given a simple set of polynomials $(R_n)_{n\geq0}$
(meaning that each $R_n\in\mathcal{P}$ and $\deg R_n=n$ for each $n=0,1,\ldots$),
the corresponding dual basis is a sequence of linear functionals
${\bf e}_n:\mathcal{P}\to\mathbb{C}$ such that
$$
\langle{\bf e}_n,R_j\rangle:=\delta_{n,j}\quad(n,j=0,1,\ldots)\;,
$$
where $\delta_{n,j}$ denotes the Kronecker's symbol.
In particular, if $(R_n)_{n\geq0}$ is a monic OPS with respect to ${\bf w}\in\mathcal{P}'$, i.e.,
there exists a sequence of nonzero complex numbers $(k_n)_{n\geq0}$ such that the orthogonality condition
$$
\langle{\bf w},R_jR_n\rangle:=k_n\delta_{j,n}\quad(j,n=0,1,\ldots)
$$
holds, then the corresponding dual basis is explicitly given by
\begin{align}\label{expression-an}
{\bf e}_n =\left\langle {\bf u} , R_n ^2 \right\rangle ^{-1} R_n{\bf w}.
\end{align}
The left multiplication of a functional ${\bf w}$ by a polynomial $f$ is defined by
$$
\left\langle f {\bf w}, p  \right\rangle =\left\langle {\bf w}, fp  \right\rangle \quad (p\in \mathcal{P}).
$$
Any functional ${\bf w} \in \mathcal{P}^*$ (when $\mathcal{P}$ is endowed with an appropriate strict inductive limit topology, see \cite{M1991}) can be written in the sense of the weak topology in $\mathcal{P}^*$ as 
\begin{align*}
{\bf w} = \sum_{n=0} ^{\infty} \left\langle {\bf w}, R_n \right\rangle {\bf e}_n.
\end{align*}
It is known that a monic OPS, $(P_n)_{n\geq 0}$, is characterized by the following three-term recurrence relation (TTRR):
\begin{align}\label{TTRR_relation}
R_{-1} (z)=0, \quad R_{n+1} (z) =(z-B_n)R_n (z)-C_n R_{n-1} (z) \quad (C_n \neq 0),
\end{align}
and, therefore,
\begin{align}\label{TTRR_coefficients}
B_n = \frac{\left\langle {\bf w} , zR_n ^2 \right\rangle}{\left\langle {\bf w} , R_n ^2 \right\rangle},\quad C_{n+1}  = \frac{\left\langle {\bf w} , R_{n+1} ^2 \right\rangle}{\left\langle {\bf w} , R_n ^2 \right\rangle}.
\end{align}

The Askey-Wilson and the averaging operators  induce two elements on $\mathcal{P}^*$, say $\mathbf{D}_q$ and $\mathbf{S}_q$, via the following definition (see \cite{FK-NM2011}): 
\begin{align*}
\langle \mathbf{D}_q{\bf w},f\rangle=-\langle {\bf w},\mathcal{D}_q f\rangle,\quad \langle\mathbf{S}_q{\bf w},f\rangle=\langle {\bf w},\mathcal{S}_q f\rangle.
\end{align*}

Hereafter we denote $z=x(s)=(q^s+q^{-s})/2$. Then the following proposition holds.
\begin{proposition}\cite{KDP2021}
Let $f,g\in\mathcal{P}$ and ${\bf w}\in\mathcal{P}^*$. Then the following equations hold.
\begin{align}
\mathcal{D}_q \big(fg\big)&= \big(\mathcal{D}_q f\big)\big(\mathcal{S}_q g\big)+\big(\mathcal{S}_q f\big)\big(\mathcal{D}_q g\big), \label{def-Dx-fg} \\[7pt]
\mathcal{S}_q \big( fg\big)&=\big(\mathcal{D}_q f\big) \big(\mathcal{D}_q g\big)\texttt{U}_2  +\big(\mathcal{S}_q f\big) \big(\mathcal{S}_q g\big), \label{def-Sx-fg} \\[7pt]
\alpha \mathcal{S}_q ^2 f&=\mathcal{S}_q \big(\texttt{U}_1 \mathcal{D}_qf\big) +\texttt{U}_2\mathcal{D}_q ^2 f +\alpha f ,\label{def-Sx^2f}\\[7pt]
%f\mathcal{D}_qg&=\mathcal{D}_q\left[ \Big(\mathcal{S}_qf-\frac{\texttt{U}_1}{\alpha}\mathcal{D}_qf \Big)g\right]-\frac{1}{\alpha}\mathcal{S}_q \Big(g\mathcal{D}_q %f\Big) , \label{def-fDxg} \\[7pt]
\mathcal{D}_q ^n\mathcal{S}_qf &=\alpha_n \mathcal{S}_q\mathcal{D}_q ^n f+\gamma_n \texttt{U}_1\mathcal{D}_q ^{n+1}f   ,\label{def-DxnSxf}\\[7pt]
f{\bf D}_q {\bf w}&={\bf D}_q\left(\mathcal{S}_qf~{\bf w}  \right)-{\bf S}_q\left(\mathcal{D}_qf~{\bf w}  \right), \label{def-fD_x-u}\\[7pt]
\alpha \mathbf{D}_q ^n \mathbf{S}_q {\bf w}&= \alpha_{n+1} \mathbf{S}_q \mathbf{D}_q^n {\bf w}
+\gamma_n \texttt{U}_1\mathbf{D}_q^{n+1}{\bf w}, \label{DxnSx-u} 
\end{align}
with $n=0,1,\ldots$, where $\alpha =(q^{1/2}+q^{-1/2})/2$ and $$\texttt{U}_1 (z)=(\alpha ^2 -1)z,\;~\texttt{U}_2 (z)=(\alpha ^2 -1)(z^2-1)\;.$$
\end{proposition}

It is known (see \cite[Proposition 2.1]{KDP2021})that 
\begin{align}
\mathcal{D}_q z^n =\gamma_n z^{n-1}+u_nz^{n-3}+\cdots,\quad \mathcal{S}_q z^n =\alpha_n z^n+\widehat{u}_nz^{n-2}+\cdots, \label{Dx-xnSx-xn}
\end{align}
with $n=0,1,\ldots$, where 
\begin{align*}
&\alpha_n= \mbox{$\frac12$}(q^{n/2} +q^{-n/2})\;,\quad \gamma_n = \frac{q^{n/2}-q^{-n/2}}{q^{1/2}-q^{-1/2}}\;,\\
&u_n= \mbox{$\frac14$}\big(n\gamma_{n-2}-(n-2)\gamma_n  \big)\;,\quad\widehat{u}_n=\mbox{$\frac n4$}(\alpha_{n-2}-\alpha_n)\;.
\end{align*}
We set $\gamma_{-1}:=-1$ and $\alpha_{-1}:=\alpha$. Recall that the monic Al-Salam-Chihara polynomials, $Q_n(x;a,b|q)$, depend on two real parameters $a$ and $b$, are characterized by %the TTRR
$$
\begin{array}{rcl}
xQ_{n}(x;a,b|q)&=&Q_{n+1}(x;a,b|q)+\,\mbox{$\frac12$}\,(a+b)q^n\,Q_{n}(x;a,b|q) \\ [0.25em]
&&\displaystyle\, +\,\mbox{$\frac14$}\,(1-ab q^{n-1})(1-q^n)\,Q_{n-1}(x;a,b;q)
\end{array}
$$
($n=0,1,\ldots$), provided we define $Q_{-1}(x;a,b|q)=0$
(see e.g. \cite{I2005}).
%(see \cite[pp.\;379 and 391]{IsmailBook2005}).
%provided we define $\widehat{P}_{-1}^{(\alpha,\beta)}(x|q):=0$.
Further, up to normalization, the Rogers $q-$Hermite polynomials are the special case $a=b=0$ of the Al-Salam-Chihara polynomials. The following result is useful.
\begin{theorem}\label{main-Thm1}\cite[Theorem 4.1]{KDP2021}
Let $(P_n)_{n\geq 0}$ be a monic OPS with respect to ${\bf w} \in \mathcal{P}^*$. 
Suppose that ${\bf u}$ satisfies the distributional equation
\begin{align}\label{x-pearson}
{\bf D}_q(\phi {\bf w})={\bf S}_q(\psi {\bf w})\;,
\end{align}
where $\phi(z)=az^2+bz+c$ and $\psi(z)=dz+e$, with $d\neq0$.
Then $(P_n)_{n\geq 0}$ satisfies \eqref{TTRR_relation} with
\begin{align}
B_n  = \frac{\gamma_n e_{n-1}}{d_{2n-2}}
-\frac{\gamma_{n+1}e_n}{d_{2n}},\quad
C_{n+1}  =-\frac{\gamma_{n+1}d_{n-1}}{d_{2n-1}d_{2n+1}}\phi^{[n]}\left( -\frac{e_{n}}{d_{2n}}\right),\label{Bn-Cn-Dx}
\end{align}
 where $d_n=a\gamma_n+d\alpha_n$, $e_n=b\gamma_n+e\alpha_n$, and \begin{align*}
\phi^{[n]}(z)=\big(d(\alpha^2-1)\gamma_{2n}+a\alpha_{2n}\big)
\big(z^2-1/2\big)+\big(b\alpha_n+e(\alpha^2-1)\gamma_n\big)z+ c+a/2\;.
\end{align*}
\end{theorem}

\section{Main results: first case}\label{first-case-main-results}
We are now in the position to prove our main results for one of the situation. 
\begin{lemma}\label{lemma-1-case-1}
Let $(P_n)_{n\geq 0}$ be a monic OPS such that 
\begin{align}\label{first-appell-type-case-SxDx}
\mathcal{S}_q\mathcal{D}_q P_{n}(z)=k_nP_{n-1}(z)\quad \quad (n=0,1,\ldots)\;.
\end{align}
Then the following relations hold:

\begin{align}
&2\alpha \texttt{U}_2(z)\mathcal{D}_q ^2P_n(z)=a_nP_n(z) +b_nP_{n-1}(z)+c_n P_{n-2}(z) \;,\label{From-first-appell-type-Dx2}\\
&4\alpha\texttt{U}_2(z) \mathcal{D}_q\mathcal{S}_q P_n(z)=\sum_{l=1} ^5 a_n ^{[l]}P_{n+2-l}(z)  \;,\label{From-first-appell-type-DxSx2}\\
&2\mathcal{S}_q ^2P_n(z)=2\alpha_n ^2 P_n(z) +k_n(B_n-B_{n-1})P_{n-1}(z)+(k_{n-1}C_n-k_nC_{n-1}) P_{n-2}(z) \;,\label{From-first-appell-type-Sx2}
\end{align}
for each $n=0,1,\ldots$, where 
\begin{align*}
a_n&=k_{n+1}-(2\alpha ^2-1)k_n-1\;,~b_n=\big( B_n-(2\alpha ^2-1)B_{n-1}\big)k_n\;,\\
c_n&=k_{n-1}C_n -(2\alpha ^2 -1)k_nC_{n-1},~a_n ^{[1]}=a_{n+1}- a_n \;,\quad a_n ^{[2]}=b_{n+1}-b_n\;,\\
a_n ^{[3]}&=c_{n+1}-c_n +(B_n-B_{n-1})b_n+(a_{n-1}-\alpha ^2a_n)C_n \;,\\
a_n ^{[4]}&=(B_n -B_{n-2})c_n+b_{n-1}C_n -b_nC_{n-1},\quad ~a_n ^{[5]}=c_{n-1}C_n- c_nC_{n-2} \;.
\end{align*}

\end{lemma}
\begin{proof}
First of all from \eqref{def-Sx^2f} using \eqref{def-Sx-fg} yields
\begin{align}\label{def-Sx^2f-to-use}
\mathcal{S}_q ^2f=\alpha \texttt{U}_2\mathcal{D}_q ^2f +\texttt{U}_1\mathcal{S}_q\mathcal{D}_qf +f\;.
\end{align} 
Secondly, we apply successively the operators $\mathcal{D}_q$ and $\mathcal{S}_q$ to the TTRR \eqref{TTRR_relation} satisfied by the monic OPS $(P_n)_{n\geq 0}$ solution of \eqref{first-appell-type-case-SxDx}. Using \eqref{def-Dx-fg} and \eqref{def-Sx-fg}, we obtain the following equation.
\begin{align}\label{star-1}
\mathcal{S}_q ^2P_n (z)+\alpha \texttt{U}_2(z)\mathcal{D}_q ^2P_n(z)&+\alpha ^2z\mathcal{S}_q\mathcal{D}_qP_n(z)\nonumber \\
&=\mathcal{S}_q\mathcal{D}_qP_{n+1}(z)+B_n\mathcal{S}_q\mathcal{D}_qP_n(z)+C_n\mathcal{S}_q\mathcal{D}_qP_{n-1}(z)\;.
\end{align}
Finally \eqref{From-first-appell-type-Dx2} is obtained from \eqref{star-1} by using successively \eqref{def-Sx^2f-to-use}, the TTRR \eqref{TTRR_relation} and \eqref{first-appell-type-case-SxDx}.
Now from \eqref{def-Sx^2f-to-use}, we may also write \eqref{star-1} as
\begin{align*}
2\mathcal{S}_q ^2P_n(z)+z\;\mathcal{S}_q\mathcal{D}_q&P_n(z)-P_n(z)\\
&=\mathcal{S}_q\mathcal{D}_qP_{n+1}(z)+B_n\mathcal{S}_q\mathcal{D}_qP_n(z)+C_n\mathcal{S}_q\mathcal{D}_qP_{n-1}(z)\;.
\end{align*}
Equation \eqref{From-first-appell-type-Sx2} is obtained from this equation using \eqref{TTRR_relation} and \eqref{first-appell-type-case-SxDx}.
\\
Lets start again with the TTRR \eqref{TTRR_relation}. We apply the operator $\mathcal{D}_q ^2$ to it using \eqref{def-Dx-fg} and \eqref{def-Sx-fg} to obtain
\begin{align}\label{for-later-use-1}
\mathcal{D}_q\mathcal{S}_qP_n(z)+\alpha \mathcal{S}_q\mathcal{D}_q&P_n(z) +\alpha ^2z\mathcal{D}_q ^2P_n(z)\nonumber\\
&=\mathcal{D}_q ^2P_{n+1}(z)+B_n\mathcal{D}_q ^2P_{n}(z)+C_n\mathcal{D}_q ^2P_{n-1}(z)\;.
\end{align}
Then \eqref{From-first-appell-type-DxSx2} is obtained by multiplying \eqref{for-later-use-1} by $2\alpha \texttt{U}_2(z)$ using successively \eqref{first-appell-type-case-SxDx}, \eqref{From-first-appell-type-Dx2} and again the TTRR \eqref{TTRR_relation}. Hence the result follows.  
\end{proof}

\begin{lemma}\label{lemma-2-case-1}
Let $(P_n)_{n\geq 0}$ be a monic OPS satisfying \eqref{first-appell-type-case-SxDx}. The following system of difference equations holds
\begin{align}
&k_{n+2}-1/2 ~-2(2\alpha ^2 -1)(k_{n+1}-1/2) +k_n -1/2 =0\;,\label{eqS1}\\
&t_{n+2} -2(2\alpha ^2-1)t_{n+1}+t_n=0,~\quad t_n:=k_n/C_n\;, \label{eqS2}\\
&k_{n+1}B_{n+1}+\big(k_{n+1}-k_{n+2}-2(2\alpha ^2 -1)k_n \big)B_n +k_nB_{n-1}=0\;,\label{eqS3}\\
&t_{n+3}B_{n+2}-(t_{n+2}+t_{n+1})B_{n+1}+t_nB_n=0\;,\label{eqS4}\\
(t_{n+1}+t_{n+2})&(C_{n+1}-1/4)-4\alpha ^2t_n(C_n-1/4)+(t_{n-1}+t_{n-2})(C_{n-1}-1/4)\nonumber \\
&\quad \quad \quad \quad \quad \quad =t_n\left[B_n ^2 -2(2\alpha ^2 -1)B_nB_{n-1}+B_{n-1} ^2  \right]\;,\label{eqS5}
\end{align}
where $B_n$ and $C_n$ are the coefficients of the TTRR \eqref{TTRR_relation} satisfied by $(P_n)_{n\geq 0}$.
\end{lemma}

\begin{proof}
Consider the TTRR \eqref{TTRR_relation} satisfied by monic OPS $(P_n)_{n\geq 0}$ solution of \eqref{first-appell-type-case-SxDx}. Then from \eqref{for-later-use-1} using \eqref{def-DxnSxf} for $n=1$ and $f=P_n$ therein, we obtain
\begin{align}\label{for-last-use-1}
2\alpha \mathcal{S}_q\mathcal{D}_qP_n(z) +(2\alpha ^2-1)z\mathcal{D}_q ^2P_n(z)
=\mathcal{D}_q ^2P_{n+1}(z)+B_n\mathcal{D}_q ^2P_{n}(z)+C_n\mathcal{D}_q ^2P_{n-1}(z)\;.
\end{align}
We now multiply \eqref{for-last-use-1} by $2\alpha \texttt{U}_2$ using successively \eqref{first-appell-type-case-SxDx}, \eqref{From-first-appell-type-Dx2} and the TTRR \eqref{TTRR_relation} to obtain a vanishing linear combination of $P_{n+1}$, $P_n$, $P_{n-1}$, $P_{n-2}$ and $P_{n-3}$, for each $n=0,1,\ldots$. Since $(P_n)_{n\geq 0}$ is a polynomial base in $\mathcal{P}$, then all coefficients of the mentioned linear combination must be zero. Therefore we obtain the following equations
\begin{align}
&a_{n+1}-(2\alpha ^2-1)a_n=4\alpha ^2(\alpha ^2 -1)k_n\;,\label{min-eq1}\\
&c_{n-1}C_n -(2\alpha ^2-1)c_nC_{n-2}=4\alpha ^2(\alpha ^2-1)k_nC_{n-1}C_{n-2}\;,\label{min-eq2}\\
&b_{n+1}-(2\alpha ^2 -1)b_n -2(\alpha ^2-1)a_nB_n=4\alpha ^2(\alpha ^2-1)(B_n+B_{n-1})k_n\;,\label{min-eq3}\\
&c_{n+1}-(2\alpha ^2-1)c_n +\big(a_{n-1}-(2\alpha ^2-1)a_n \big)C_n+\big(B_n-(2\alpha ^2-1)B_{n-1} \big)b_n\nonumber \\
&\quad \quad \quad \quad\quad \quad \quad \quad \quad\quad \quad \quad\quad=4\alpha ^2(\alpha ^2-1)k_n(C_n+B_{n-1} ^2+C_{n-1}-1) \;,\label{min-eq4}\\
&b_{n-1}C_n +\big(B_n-(2\alpha ^2-1)B_{n-2} \big)c_n -(2\alpha ^2-1)b_nC_{n-1}\nonumber\\
&\quad\quad \quad\quad \quad \quad \quad\quad \quad \quad \quad\quad \quad \quad  =4\alpha ^2(\alpha ^2-1)k_nC_{n-1}(B_{n-1}+B_{n-2}) \;.\label{min-eq5}
\end{align}
Equations \eqref{eqS1} and \eqref{eqS2} follow from \eqref{min-eq1} and \eqref{min-eq2}, respectively using notations and expressions of $a_n$, $b_n$ and $c_n$ obtained in the previous lemma. Similarly, \eqref{eqS3}--\eqref{eqS5} are obtained from \eqref{min-eq3}--\eqref{min-eq5} using \eqref{eqS1} and \eqref{eqS2}.
\end{proof}

\begin{theorem}\label{main-theo-for-SxDx}
The only monic OPS, $(P_n)_{n\geq 0}$, for which  
\begin{align}
\mathcal{S}_q\mathcal{D}_qP_{n}(z)= k_nP_{n-1}(z) \quad \quad (n=0,1,\ldots)\;,\label{appell-case-SxDx-solved}
\end{align}
is the Al-Salam-Chihara polynomial with parameters $a$ and $b$ such that $(a,b)\in \left\lbrace (1,-1),~(-1,1)  \right\rbrace$.
\end{theorem}  

\begin{proof}
Let $(P_n)_{n\geq 0}$ be a monic OPS solution of  %with respect to the functional ${\bf u} \in \mathcal{P}^*$ and satisfying 
\eqref{appell-case-SxDx-solved}. Before solving the system of equations \eqref{eqS1}--\eqref{eqS5}, let us find some initial conditions.
We claim that the coefficients $B_n$ and $C_n$ of the TTRR \eqref{TTRR_relation} satisfied by $(P_n)_{n\geq 0}$ are given by
\begin{align}
&B_{n-1}=0,\quad k_{n-1}=\gamma_{n-1}\alpha_{n-2}\;,\label{initial-conditions-first-appell-SxDx-01}\\
&C_{n+1}-1/4 = \frac{k_n-k_{n+2}+\gamma_{n+2}\alpha_{n-1}}{4k_n} +\frac{k_{n+2}-k_n}{k_n}\sum_{l=1} ^n (C_l-1/4)\;,\label{initial-conditions-first-appell-SxDx-02}
\end{align}
for each $n=1,2,\ldots$.
Indeed, it is known that 
$P_n(z)=z^n +f_nz^{n-1}+g_nz^{n-2}+\cdots$, where $f_0=g_1=0$, $B_n=f_n-f_{n+1}$ and $C_n=g_n-g_{n+1}-f_nB_n$. With this we identify the three first coefficients of term with higher degrees in \eqref{appell-case-SxDx-solved} using \eqref{Dx-xnSx-xn} to obtain
$k_n=\gamma_n\alpha_{n-1}$ together with 
\begin{align}\label{expr-fn-gn}
k_{n-1}f_n=k_nf_{n-1},~\quad k_ng_{n-1}=k_{n-2}g_n +\gamma_n\widehat{u}_{n-1}+\alpha_{n-3}u_n\;.
\end{align} 
From the first equation in \eqref{expr-fn-gn} we obtain $$B_n=(k_{n+1}-k_n)B_0\;,$$
which also satisfies \eqref{eqS3}. In addition, assume without loss of generality that $0<q<1$. Then $$\lim_{n\rightarrow \infty }q^nB_n=B_0/2 \;.$$ It is not hard to see that $q$ and $q^{-1}$ are solutions of the characteristic equation associated to \eqref{eqS2}. Hence solutions of the mentioned equation are given by $$t_n=r_1q^n +r_2q^{-n}~\quad (n=1,2,\ldots)\;,$$ with $r_1$ and $r_2$ two complex numbers such that $|r_1|+|r_2|\neq 0$.
Assume for instance that $r_2\neq 0$. Then we write $t_n=r_2q^{-n}(1-rq^{2n})$, where $r=-r_1/r_2$. We multiply \eqref{eqS4} by $q^n$ and take the limit as $n$ tends to $\infty$ to obtain $B_0=0$ and therefore $B_n=0$, for all $n=0,1,\ldots$. So \eqref{initial-conditions-first-appell-SxDx-01} holds and the second equation in \eqref{initial-conditions-first-appell-SxDx-02} is then obtain directly from the second equation in \eqref{expr-fn-gn}.

From the definition of $t_n$ given in \eqref{eqS2}, we obtain $C_n=k_n/t_n$ and so we deduce that $\lim_{n\rightarrow \infty} C_n =1/(2r_2(q^{-1}-1))$. But taking the limit in \eqref{eqS5} as $n$ tends to $\infty$ taking into account \eqref{initial-conditions-first-appell-SxDx-01}, we obtain $\lim_{n\rightarrow \infty}C_n=1/4$. This means we can write
\begin{align*}
C_{n}=\frac{(1-q^n)(1+q^{n-1})}{4(1-rq^{2n})}\;\quad (n=1,2,\ldots)\;.
\end{align*}
It is not hard to see that this satisfy \eqref{initial-conditions-first-appell-SxDx-02} and \eqref{eqS5} if and only if $r=0$ and therefore $C_n=(1-q^n)(1+q^{n-1})/4$. For the case $1<q<+\infty$. We proceed similarly to obtain $C_n=(1-q^{-n})(1+q^{-n+1})/4$.
Hence solutions of \eqref{appell-case-SxDx-solved} are given by $B_n=0$ with 
$$C_{n+1}=\mbox{$\frac14$}(1-q^{n+1})(1+q^n)\quad\textit{or}\quad C_{n+1}=\mbox{$\frac14$}(1-q^{-n-1})(1+q^{-n})\;,$$ for all $n=0,1,\ldots$. Thus $$P_n=Q_n(z;s,-s|q) \quad\textit{or}\quad P_n=Q_n(z;s,-s|1/q),\quad s=\pm 1\;.$$
\end{proof}
We now characterize functionals whose corresponding OPS are solutions of \eqref{appell-case-SxDx-solved}.

\begin{theorem}\label{main-result-3}
Let ${\bf u}\in \mathcal{P}^*$ be a regular functional and $(P_n)_{n\geq 0}$ the corresponding monic OPS. Then $(P_n)_{n\geq 0}$ is the solution of \eqref{appell-case-SxDx-solved} if and only if ${\bf u}$ is a solution of the following functional equations
\begin{align}
(q^s-1){\bf D}_q{\bf S}_q{\bf u}&=2z{\bf u}\;,\label{funct-eq1}\\
2q^{3s/2}{\bf D}_q ^2 \big(\texttt{U}_2{\bf u} \big)&=-(2z^2+q^s-1){\bf u}\;,\label{funct-eq2}\\
2q^s{\bf S}_q ^2{\bf u}&=\big(-2z^2 +1+q^s  \big){\bf u}\;,\label{funct-eq3}\\
8q^{5s/2}{\bf S}_q {\bf D}_q\big( \texttt{U}_2 {\bf u} \big) &=(1-q^s)z(-4z^2+q^{2s}+3){\bf u}\;,\label{funct-eq4}
\end{align} 
with $s=\pm 1$.
\end{theorem}

\begin{remark}
We emphasize the following. At this stage we know that monic OPS solutions of \eqref{appell-case-SxDx-solved} are special cases the Al-Salam-Chihara polynomials and so they are classical OPS. Then, there exit (see \cite{FK-NM2011}) two polynomials $\phi$ and $\psi$, of degree at most two and one, respectively such that $${\bf D}_q(\phi {\bf u})={\bf S}_q(\psi {\bf u})\;.$$
However, from the above functional equation it is not possible to deduce \eqref{funct-eq1}--\eqref{funct-eq4}. Nevertheless, using Lemma \ref{lemma-1-case-1} the result can be proved as follows.
\end{remark}

\begin{proof}
Assume first that $(P_n)_{n\geq 0}$ is the monic OPS solution of \eqref{appell-case-SxDx-solved}. Let $({\bf a}_n)_{n\geq 0}$ be the dual basis associated to the sequence of simple set $(P_n)_{n\geq 0}$. Using \eqref{appell-case-SxDx-solved}, the following holds
$$\left\langle {\bf D}_q{\bf S}_q{\bf a}_n, P_l \right\rangle =-\left\langle {\bf a}_n, \mathcal{D}_q \mathcal{S}_q P_l \right\rangle=-k_l\left\langle {\bf a}_n, P_{l-1}\right\rangle=-k_{n+1}\delta_{n+1,l}\;.$$
Therefore
\begin{align}
{\bf D}_q{\bf S}_q{\bf a}_n=-k_{n+1}{\bf a}_{n+1}\;\quad~(n=0,1,\ldots)\;,\label{funct-eq1-for-an}
\end{align}
is obtained by writing $${\bf D}_q{\bf S}_q{\bf a}_n =\sum_{l=0} ^{+\infty} \left\langle {\bf D}_q{\bf S_q}{\bf a}_n, P_l \right\rangle{\bf a}_l\;, $$ taking into account what is preceding. Equation \eqref{funct-eq1} follows by taking $n=0$ in \eqref{funct-eq1-for-an} using \eqref{expression-an}, \eqref{TTRR_coefficients} and the fact that $B_0=0$ and $C_1=(1-q^s)/2$ with $s=\pm 1$ (obtained from Theorem \ref{main-theo-for-SxDx}). Similarly, using \eqref{From-first-appell-type-Dx2} on can prove that 
\begin{align*}
2\alpha {\bf D}_q ^2(\texttt{U}_2{\bf a}_n)= a_n{\bf a}_n +b_{n+1}{\bf a}_{n+1}+c_{n+2}{\bf a}_{n+2}\; \quad \quad (n=0,1,\ldots)\;.
\end{align*}
Therefore \eqref{funct-eq2} follows by taking $n=0$ in the above equation taking into account \eqref{TTRR_relation}--\eqref{TTRR_coefficients}, \eqref{expression-an} and Theorem \ref{main-theo-for-SxDx}. Equation \eqref{funct-eq3} (respectively \eqref{funct-eq4}) follows from the same idea using \eqref{From-first-appell-type-Sx2} (respectively \eqref{From-first-appell-type-DxSx2}). 
\\
Assume secondly that $(P_n)_{n\geq 0}$ is a monic OPS with respect to the functional ${\bf u}$, solution of equations \eqref{funct-eq1}--\eqref{funct-eq4}. We are only going to use \eqref{funct-eq1} and \eqref{funct-eq3}. We first apply the operator ${\bf S}_q$ on \eqref{funct-eq1} using successively \eqref{DxnSx-u} (for $n=1$ and ${\bf u}$ replaced by ${\bf S}_q{\bf u}$), \eqref{funct-eq3} and \eqref{funct-eq1} to obtain
\begin{align}\label{relation-for-converse}
{\bf S}_q (z{\bf u})&=\frac{1}{2}(q^s-1){\bf S}_q{\bf D}_q\big({\bf S}_q{\bf u}\big)=\frac{1}{2}(q^s-1)\Big(\frac{\alpha}{2\alpha ^2-1} {\bf D}_q{\bf S}_q ^2{\bf u}-\frac{1}{2\alpha ^2-1}\texttt{U}_1{\bf D}_q ^2{}{\bf S}_q{\bf u}\Big)\nonumber\\
&=\frac{\alpha (q^s-1)}{2(2\alpha ^2-1)}{\bf D}_q\Big(\big(-q^{-s}z^2+\mbox{$\frac12$}(1+q^{-s}) \big){\bf u} \Big)-\frac{\texttt{U}_1}{2\alpha ^2-1}{\bf D}_q\big(z{\bf u}  \big)\;.
\end{align}
In the meantime, using \eqref{def-fD_x-u} one may write $$\texttt{U}_1{\bf D}_q\big(z{\bf u} \big)=\alpha {\bf D}_q \big(z\texttt{U}_1{\bf u} \big)-(\alpha ^2-1){\bf S}_q ^2\big( z{\bf u}\big)  \;.$$
We replace this in \eqref{relation-for-converse} in order to obtain
$$\big(q^{1/2}-q^{-1/2}\big){\bf D}_q\big((z^2-1){\bf u} \big)=-2s{\bf S}_q\big(z{\bf u}\big)\;.$$
This means that ${\bf u}$ satisfies \eqref{x-pearson} with $\phi(z)=\mbox{$\frac s2$} (q^{1/2}-q^{-1/2})(z^2-1)$ and $\psi(z)=z$. Therefore applying \eqref{Bn-Cn-Dx}, we obtain
$$B_n =0,\quad C_{n+1}=\frac{(1-q^{s(n+1)})(1+q^{sn})}{4}\quad (n=0,1,\ldots)\;,$$
and so $P_n =Q_n(x;s,-s|q^s)$ for $n=0,1,\ldots$. We then use Theorem \ref{main-theo-for-SxDx} to conclude that $(P_n)_{n\geq 0}$ satisfies \eqref{appell-case-SxDx-solved}. This conclusion can be obtained similarly using \eqref{funct-eq2} and \eqref{funct-eq4}.
\end{proof}

\section{Main results: second case}\label{second-case-main-results}
In this section we are interested in monic OPS, $(P_n)_{n\geq 0}$, solution of the following equation
\begin{align}\label{second-appell-type-case-DxSx}
\mathcal{D}_q\mathcal{S}_q P_{n}(z)=r_nP_{n-1}(z)\quad \quad (n=0,1,\ldots)\;.
\end{align}
Methods and techniques are similar to ones used in the previous section. For this reason, we mention some of results without proves.   

\begin{lemma}\label{lemma-1-case-2}
Let $(P_n)_{n\geq 0}$ be a monic OPS such that \eqref{second-appell-type-case-DxSx} holds. 
Then the following relations hold:
\begin{align}
2(\alpha ^2-1)(z^2-\alpha ^2) \mathcal{D}_q ^2P_n(z)=a_nP_n(z) +b_nP_{n-1}(z)+c_n P_{n-2}(z) \;,\label{From-second-appell-type-Dx2}\\
2\alpha \mathcal{S}_q ^2P_n(z)=c_n ^{[1]}P_n(z) +c_n ^{[2]}P_{n-1}(z)+c_n ^{[3]}P_{n-2}(z)\;,\label{From-second-appell-type-Sx2}\\
4\alpha (\alpha ^2-1)(z^2-\alpha ^2) \mathcal{S}_q\mathcal{D}_q P_n(z)\quad\quad\quad\quad\quad \quad\quad\quad\quad\quad\quad\quad\quad\quad\quad\quad\quad\quad\quad\quad\nonumber\\
 =b_n ^{[1]}P_{n+1}(z) +b_n ^{[2]}P_{n}(z)+b_n ^{[3]} P_{n-1}(z) +b_n ^{[4]}P_{n-2}(z)+b_n ^{[5]}P_{n-3}(z) \;,\label{From-second-appell-type-DxSx2}
\end{align}
for each $n=0,1,\ldots$, where 
\begin{align*}
a_n&=r_{n+1}-(4\alpha ^2-3)r_n-\alpha\;,~b_n=\big( B_n-(4\alpha ^2-3)B_{n-1}\big)r_n\;,\\
c_n&=r_{n-1}C_n -(4\alpha ^2 -3)r_nC_{n-1},~b_n ^{[1]}=a_{n+1}-(2\alpha ^2-1) a_n \;,\\
b_n ^{[2]}&=b_{n+1}-(2\alpha ^2-1)b_n -2(\alpha ^2-1)a_nB_n\;,\\
b_n ^{[3]}&=c_{n+1}-(2\alpha ^2-1)c_n +(B_n-(2\alpha ^2-1)B_{n-1})b_n+(a_{n-1}-(2\alpha ^2-1)a_n)C_n \;,\\
b_n ^{[4]}&=(B_n -(2\alpha ^2-1)B_{n-2})c_n+b_{n-1}C_n -(2\alpha ^2-1)b_nC_{n-1}, \\
b_n ^{[5]}&=c_{n-1}C_n-(2\alpha ^2-1) c_nC_{n-2},\quad c_n ^{[1]}=r_{n+1}-(2\alpha ^2-1)r_n+\alpha\;,\\
c_n ^{[2]}&=\big(B_n-(2\alpha ^2-1)B_{n-1} \big)r_n,\quad c_n ^{[3]}=r_{n-1}C_n-(2\alpha ^2-1)r_nC_{n-1}\;.
\end{align*}
\end{lemma}

\begin{proof}
From \eqref{def-Sx^2f} using \eqref{def-Sx-fg} yields
\begin{align}\label{def-Sx^2f-to-use2}
\alpha \mathcal{S}_q ^2f=\big(\alpha^2 \texttt{U}_2 -\texttt{U}_1 ^2 \big)\mathcal{D}_q ^2f +\texttt{U}_1\mathcal{D}_q\mathcal{S}_qf +\alpha f\;.
\end{align} 
We apply successively the operators $\mathcal{S}_q$ and $\mathcal{D}_q$ to the TTRR \eqref{TTRR_relation} satisfied by the monic OPS $(P_n)_{n\geq 0}$ solution of \eqref{second-appell-type-case-DxSx}. Using \eqref{def-Dx-fg}, \eqref{def-Sx-fg} and \eqref{def-Sx^2f-to-use2}, we obtain the following equation.
\begin{align}\label{star-1-2}
\big(\alpha ^2z+3\texttt{U}_1(z) \big)\mathcal{D}_q& \mathcal{S}_q P_n (z)+2\big(\alpha ^2 \texttt{U}_2(z) -\texttt{U}_1 ^2(z)\big)\mathcal{D}_q ^2P_n(z)+\alpha P_n(z)\nonumber \\
&=\mathcal{D}_q\mathcal{S}_qP_{n+1}(z)+B_n\mathcal{D}_q\mathcal{S}_qP_n(z)+C_n\mathcal{D}_q\mathcal{S}_qP_{n-1}(z)\;,
\end{align}
since $\mathcal{D}_q\texttt{U}_2=2\alpha\texttt{U}_1$ and $\mathcal{S}_q\texttt{U}_2=\alpha ^2\texttt{U}_2 +\texttt{U}_1 ^2$.
Finally \eqref{From-second-appell-type-Dx2} is obtained from \eqref{star-1-2} by using successively \eqref{second-appell-type-case-DxSx} and the TTRR \eqref{TTRR_relation}.
Now from \eqref{def-Sx^2f-to-use2}, we may also write \eqref{star-1-2} as
\begin{align*}
2\alpha \mathcal{S}_q ^2P_n(z)&-\alpha P_n(z)+(2\alpha ^2-1)z\;\mathcal{D}_q\mathcal{S}_qP_n(z) \\
&=\mathcal{D}_q\mathcal{S}_qP_{n+1}(z)+B_n\mathcal{D}_q\mathcal{S}_qP_n(z)+C_n\mathcal{D}_q\mathcal{S}_qP_{n-1}(z)\;.
\end{align*}
Equation \eqref{From-second-appell-type-Sx2} is obtained from this equation using \eqref{TTRR_relation} and \eqref{second-appell-type-case-DxSx}. Equation \eqref{From-second-appell-type-DxSx2} is obtained by multiplying \eqref{for-later-use-1} by $2(\alpha ^2-1)(z^2-\alpha ^2)$ using successively \eqref{second-appell-type-case-DxSx}, \eqref{From-second-appell-type-Dx2} and again the TTRR \eqref{TTRR_relation}.  
\end{proof}

\begin{lemma}
Let $(P_n)_{n\geq 0}$ be a monic OPS satisfying \eqref{second-appell-type-case-DxSx}. The following system of difference equations holds
\begin{align}
&r_{n+2} ~-2(2\alpha ^2 -1)r_{n+1} +r_n  =0\;,\label{eqS11}\\
&t_{n+2} -2(2\alpha ^2-1)t_{n+1}+t_n=0,~\quad t_n:=r_n/C_n\;, \label{eqS22}\\
&r_{n+1}B_{n+1}-(4\alpha ^2 -3)(r_n+r_{n+1})B_n +r_nB_{n-1}=0\;,\label{eqS33}\\
&t_{n+3}B_{n+2}-(t_{n+2}+t_{n+1})B_{n+1}+t_{n}B_{n}=0\;,\label{eqS44}\\
t_{n+2}&(C_{n+1}-1/4)-2t_n(C_n-1/4)+t_{n-2}(C_{n-1}-1/4)\nonumber \\
&\quad \quad \quad \quad \quad \quad =t_n\left[B_n ^2 -2(2\alpha ^2 -1)B_nB_{n-1}+B_{n-1} ^2  \right]\;,\label{eqS55}
\end{align}
where $B_n$ and $C_n$ are the coefficients of the TTRR \eqref{TTRR_relation} satisfied by $(P_n)_{n\geq 0}$.
\end{lemma}

\begin{proof}
As in the proof of Lemma \ref{lemma-2-case-1}, we multiply \eqref{for-last-use-1} by $2(\alpha ^2-1)(z^2-\alpha ^2)$ using \eqref{def-DxnSxf} for $n=1$ to obtain
\begin{align}
4(\alpha ^2 -1)(z^2-\alpha ^2)&\mathcal{D}_q\mathcal{S}_qP_n(z) +2(\alpha ^2 -1)(z^2-\alpha ^2)z\mathcal{D}_q ^2P_n(z)\nonumber \\
&=2(\alpha ^2 -1)(z^2-\alpha ^2)\mathcal{D}_q ^2 \big(P_{n+1}(z)+B_nP_n(z)+C_nP_{n-1}(z)\big)\;.\label{interm-last-case-01}
\end{align}
Also, from the TTRR, on may obtain the following relation
\begin{align*}
(z^2 -\alpha ^2)&P_n(z)= P_{n+1}(z)+(B_n+B_{n-1})P_n(z)+C_{n-1}C_{n-2}P_{n-3}(z)\\
&+(C_n+B_{n-1} ^2 +C_{n-1} -\alpha ^2)P_{n-1}(z)+(B_{n-1}+B_{n-2})C_{n-1}P_{n-2}(z)\;.
\end{align*}
Taking into account this equation together with \eqref{second-appell-type-case-DxSx}, \eqref{From-second-appell-type-Dx2} and the TTRR \eqref{TTRR_relation}, \eqref{interm-last-case-01} becomes a vanishing linear combination of polynomials $P_{n+1}$, $P_n$, $P_{n-1}$, $P_{n-2}$ and $P_{n-3}$, for each $n$. Since this is a base of the space $\mathcal{P}$, all coefficients of the mentioned linear combination are zero and so the following system of equations holds. 
\begin{align*}
&4(\alpha ^2-1)r_n +a_n =a_{n+1}\;,\\
&4(\alpha ^2 -1)r_n(B_n+B_{n-1}) +b_n=b_{n+1}\;,\\
&4(\alpha ^2-1)r_n(C_n+B_{n-1} ^2 -\alpha ^2) +(a_n-a_{n-1})C_n+b_n(B_{n-1}-B_n)=c_{n+1}-c_n\;,\\
&4(\alpha ^2-1)r_n(B_{n-1}+B_{n-2})C_{n-1}+b_nC_{n-1}-b_{n-1}C_n=c_n(B_n-B_{n-2})\;,\\
&4(\alpha ^2-1)r_nC_{n-1}C_{n-2}+c_nC_{n-2}=c_{n-1}C_n\;.
\end{align*}
The result follows.
\end{proof}

\begin{theorem}\label{main-theo-for-DxSx}
The only monic OPS, $(P_n)_{n\geq 0}$, for which  
\begin{align}
\mathcal{D}_q\mathcal{S}_qP_{n}(z)= k_nP_{n-1}(z) \quad \quad (n=0,1,\ldots)\;,\label{appell-case-DxSx-solved}
\end{align}
is the Rogers $q^2$-Hermite or Rogers $q^{-2}$-Hermite polynomial.
\end{theorem}  

\begin{proof}
Let $(P_n)_{n\geq 0}$ be a monic OPS solution of  %with respect to the functional ${\bf u} \in \mathcal{P}^*$ and satisfying 
\eqref{appell-case-DxSx-solved}. Following the proof of Theorem \ref{main-theo-for-SxDx} we obtain 
\begin{align*}
k_n=\gamma_{2n}/2~,\quad B_n=B_0\alpha_{2n+1}/\alpha\;,
\end{align*}
by identifying the two first coefficients of terms with higher degrees in \eqref{appell-case-DxSx-solved}. We then show that this expression of $B_n$ satisfies \eqref{eqS33} if and only if $B_0=0$, and so $B_n=0$, for each $n$. Finally from \eqref{eqS22} and \eqref{eqS44}, we deduce
$$B_n=0,\;\quad C_{n+1}=\mbox{$\frac14$}(1-q^{2s(n+1)})\quad (n=0,1,\ldots)\;,$$ and so we obtain $$P_n=Q_n\big(x;sq^{s/2},-sq^{s/2}|q^s \big)\quad s=\pm 1,\quad (n=0,1,\ldots)\;. $$ Hence the result follows. 
\end{proof}

\begin{theorem}
Let ${\bf u}\in \mathcal{P}^*$ be a regular functional and $(P_n)_{n\geq 0}$ the corresponding monic OPS. Then $(P_n)_{n\geq 0}$ is the solution of \eqref{appell-case-DxSx-solved} if and only if ${\bf u}$ is a solution of the following functional equations
\begin{align}
q^{s/2}(q^s-1){\bf S}_q{\bf D}_q{\bf u}&=2z{\bf u}\;,\label{funct-eq11}\\
4(\alpha ^2-1)q^{5s/2}{\bf D}_q ^2 \big((z^2-\alpha ^2){\bf u} \big)&=(-4z^2+1-q^{2s}){\bf u}\;,\label{funct-eq22}\\
4q^{2s}{\bf S}_q ^2{\bf u}&=\big(-4z^2 +1+3q^{2s}  \big){\bf u}\;,\label{funct-eq33}\\
2q^{3s}(1-q^s){\bf D}_q {\bf S}_q\big( (z^2-\alpha ^2) {\bf u} \big) &=z(-4z^2+q^{4s}+q^{2s}+2){\bf u}\;,\label{funct-eq44}
\end{align} 
with $s=\pm 1$.
\end{theorem}
\begin{proof}
This result follows proceeding exactly as in Theorem \ref{main-result-3}.
\end{proof}

\begin{remark}
Although the results obtained here were proved for the $q$-quadratic lattices, they can be easily extended to quadratic lattices $x(s)=\mathfrak{c}_4s^2+\mathfrak{c}_5s+\mathfrak{c}_6$ by taking the appropriate limit as it was discussed in \cite{KDP2021}. 
\end{remark}

%\section*{Disclosure statement}
%No potential conflict of interest was reported by the author(s).

\section*{Acknowledgements}
We would like to thank Kenier Castillo for drawing our attention to this problem. The author D. Mbouna was partially supported by CMUP, member of LASI, which is financed by national funds through FCT - Fundac\~ao para a Ci\^encia e a Tecnologia, I.P., under the projects with reference UIDB/00144/2020 and UIDP/00144/2020. A. Suzuki is supported by the FCT grant 2021.05089.BD and partially supported by the Centre for Mathematics of the University of Coimbra-UIDB/00324/2020, funded by the Portuguese Government through FCT/ MCTES.

{

\end{document}
\begin{thebibliography}{99}

%%%%%%%%%%%%%%%% pour moi %%%%%%%%%%%%%%%%%%%


\bibitem{A-1995} W. Al-Salam, A characterization of the Rogers $q$-Hermite polynomials, Internat. J. Math. and Math. Sci. \textbf{18} (1995), no. 4, 641--648.

\bibitem{Al-Salam-1972} W. Al-Salam and T. S. Chihara, Another characterization of the classical orthogonal polynomials, SIAM J. Math. Anal. \textbf{3} (1972) 65--70.

\bibitem{A-1967} {W. Al-Salam}, {$q$-Appell polynomials}, Ann. Mat. Pura Appl. {\bf vol 77 4} (1967), pp. 31-45.

\bibitem{RKDP2020} {R. \'Alvarez-Nodarse, K. Castillo, D. Mbouna, and J. Petronilho},
{On discrete coherent pairs of measures}, J. Difference Equ. Appl., {\bf vol. 28, no. 7} (2022) 853-868;

\bibitem{PA1880} {P. Appell}, {Sur une classe de polyn\^omes}, Ann. Sci. de l'Ecole Norm. Sup. {\bf (2) 9} (1880) 119-144.

\bibitem{ALPM2008} {F. Ana Loureiro and P. Maroni}, {Quadratic decomposition of Appell sequences}, Expo. Math. {\bf 26} (2008) 177-186. 


\bibitem{KCDMJP2021-a} K. Castillo, D. Mbouna, and J. Petronilho, {Remarks on Askey-Wilson polynomials and Meixner polynomials of the second kind}, Ramanujan J., {\bf 58} (2022) 1159-1170.

\bibitem{KDP2021}
K. Castillo, D. Mbouna, and J. Petronilho, On the functional equation for classical orthogonal polynomials on lattices, J. Math. Anal. Appl. {\bf 515} (2022) 126390.

\bibitem{KDPconj}
K. Castillo, D. Mbouna, and J. Petronilho, {A characterization of continuous q-Jacobi, Chebyshev of the first kind and Al-Salam Chihara polynomials}, J. Math. Anal. Appl. {\bf 514} (2022) 126358.

\bibitem{CFZ2019}{Y. Chen, G. Filipuk and L. Zhan}, {Orthogonal polynomials, asymptotics, and Heun equations}, J. Math. Phys. {\bf 60}, 113501 (2019)

\bibitem{DB2014} {J. Daniel Galiffa and W. Boon Ong}, {A characterization of an Askey-Wilson difference equation}, 
J. Difference Equ. Appl., {\bf vol. 20, no. 9} (2014) 1372-1381;

\bibitem{Griff2006} {S. Datta and J. Griffin}, {A characterization of some q-orthogonal polynomials}, Ramanujan J. {\bf 12}(2006), pp. 425–437.


\bibitem{DFS2021}{A. Dzhamay, G. Filipuk and A. Stokes}, {On differential systems related to generalized Meixner and deformed Laguerre orthogonal polynomials}, Integral Transforms Spec. Funct., {\bf 32:5-8} (2021), 483-492.

\bibitem{FR2018}{G. Filipuk and M.N. Rebocho}, {Differential equations for families of semi-classical orthogonal polynomials within class one}, Appl. Numer. Math., {\bf 124} (2018) 76-88.


\bibitem{FK-NM2011} {M. Foupouagnigni, M. Kenfack-Nangho, and S. Mboutngam},
{Characterization theorem of classical orthogonal polynomials
on nonuniform lattices: the functional approach},
Integral Transforms Spec. Funct. {\bf 22} (2011) 739-758.


\bibitem{I2005}
{M. E. H. Ismail},
{Classical and quantum orthogonal polynomials in one variable. With two chapters by W. Van Assche.
With a foreword by R. Askey.}, Encyclopedia of Mathematics and its Applications {\bf 98},
Cambridge University Press, Cambridge, 2005.

\bibitem{M1991}
{P. Maroni}, Une th\'eorie alg\'ebrique des polyn\^omes orthogonaux. Applications aux polyn\^omes orthogonaux semiclassiques, In C. Brezinski et al. Eds., Orthogonal Polynomials and Their Applications, Proc. Erice 1990, IMACS, Ann. Comp. App. Math. {\bf 9} (1991) 95-130.

\bibitem{AS1954} {A. Sharma, A. Chak}, {The basic analogue of a class of polynomials}, Revisita di Matematica della Universit\'a di Parma, {\bf vol 5} (1954) pp. 15-38.

\end{thebibliography}
